\PassOptionsToPackage{table,xcdraw}{xcolor}

\documentclass[10pt]{amsart}

\makeatletter
\def\ifdraft{\ifdim\overfullrule>\z@
\expandafter\@firstoftwo\else\expandafter\@secondoftwo\fi}
\makeatother

\usepackage{amsmath, amsthm, thmtools, marginnote, ltxcmds, adjustbox, graphicx, stmaryrd}

\usepackage{mathtools, stackengine, scalerel}

\usepackage[nice]{nicefrac}

\usepackage[normalem]{ulem}

\usepackage{xr}
\newcommand{\crefnolink}[1]{\cref*{#1}}

\usepackage{todonotes}

\makeatletter
\newcommand*{\saved@uline}{}
\let\saved@uline\uline

\newcommand*{\mathuline}{%
  \mathpalette{\math@uline\saved@uline}%
}
\newcommand*{\math@uline}[3]{%
  \mbox{#1{$#2#3\m@th$}}%
}

\renewcommand*{\uline}{%
  \relax
  \ifmmode
  \expandafter\mathuline%
  \else
  \expandafter\saved@uline%
  \fi
}
\makeatother

\usepackage[style=alphabetic, backref=true, backend=biber, hyperref=true, giveninits=true]{biblatex}

\AtEveryBibitem{
  \clearfield{urlyear}
  \clearfield{urlmonth}
}

\AtBeginBibliography{\small}

\renewbibmacro*{pageref}{%
  \iflistundef{pageref}
  {}
  {%
\printtext{\addperiod\space $\uparrow${\printlist[pageref][-\value{listtotal}]{pageref}}}}}

\DeclareLabelalphaTemplate{
  \labelelement{
    \field[final]{shorthand}
    \field{label}
    \field[strwidth=1,strside=left,ifnames=1]{labelname}
    \field[strwidth=1,strside=left]{labelname}
  }
}

\definecolor{cite}{HTML}{bb3e03}
\definecolor{url}{HTML}{698996}
\definecolor{link}{HTML}{005f73}

\usepackage[pdfencoding=unicode, colorlinks=true, linkcolor=link, citecolor=cite, urlcolor=url, linktocpage]{hyperref}
\usepackage[hyphenbreaks]{breakurl}
\usepackage{xurl}
\hypersetup{breaklinks=true}

\hypersetup{final}

\usepackage{tikz}
\usetikzlibrary{cd}
\usetikzlibrary{arrows, arrows.meta, positioning, calc}
\tikzcdset{arrow style=tikz, diagrams={>={Straight Barb[scale=0.8]}}}

\tikzstyle{arrow} = [-{Straight Barb[scale=0.8]}, line width=0.2mm]

\usepackage{enumitem}
\setlist[itemize,1]{label=--}
\setlist[itemize,2]{label=+}
\setlist[itemize,3]{label=$\bullet$}
\setlist[itemize,4]{label=$\circ$}

\setlist[enumerate,1]{label=(\roman*)}

\setlist{nosep}

\usepackage[
  letterpaper,
  twoside=true,
  textheight=22cm,
  textwidth=15cm,
  marginparsep=0.75cm,
  marginparwidth=2.5cm,
  heightrounded,
  centering
]{geometry}
\usepackage[protrusion=true]{microtype}

\usepackage{iftex}
\ifLuaTeX
\usepackage{fontspec}
\usepackage[math-style=TeX]{unicode-math}

\else
\usepackage[bitstream-charter]{mathdesign}
\usepackage[T1]{fontenc}
\fi

\usepackage{biolinum}

\renewcommand{\mathsf}[1]{\text{\normalfont\sffamily#1}}

\DeclareMathAlphabet{\eur}{U}{zeus}{m}{n}
\renewcommand{\mathcal}[1]{\eur{#1}}

\usepackage[capitalise]{cleveref}

\Crefname{prop}{Proposition}{Propositions}
\Crefname{lem}{Lemma}{Lemmas}
\Crefname{cor}{Corollary}{Corollaries}
\Crefname{thm}{Theorem}{Theorems}

\Crefname{defn}{Definition}{Definitions}
\Crefname{notation}{Notation}{Notations}
\Crefname{conj}{Conjecture}{Conjectures}
\Crefname{ass}{Assumption}{Assumptions}
\Crefname{expt}{Expectation}{Expectations}

\Crefname{rmk}{Remark}{Remarks}
\Crefname{question}{Question}{Questions}
\Crefname{expl}{Example}{Examples}

\Crefname{figure}{Figure}{Figures}

\crefformat{equation}{(#2#1#3)}
\crefmultiformat{equation}{(#2#1#3)}{ and~(#2#1#3)}{, (#2#1#3)}{, and~(#2#1#3)}
\crefformat{section}{\S#2#1#3}
\crefmultiformat{section}{\S\S#2#1#3}{ \&~#2#1#3}{, #2#1#3}{, \&~#2#1#3}

\theoremstyle{plain}
\newtheorem{prop}[subsubsection]{Proposition}
\newtheorem{lem}[subsubsection]{Lemma}
\newtheorem{cor}[subsubsection]{Corollary}
\newtheorem{thm}[subsubsection]{Theorem}
\newtheorem*{thm*}{Theorem}

\theoremstyle{definition}

\theoremstyle{remark}
\newtheorem{rmk}[subsubsection]{Remark}
\newtheorem*{rmk*}{Remark}

\numberwithin{equation}{subsection}

\newcommand{\teq}{\addtocounter{subsubsection}{1}\tag{\thesubsubsection}}

\DeclareMathOperator{\BiMod}{\mathsf{BiMod}}
\newcommand{\Br}{\mathrm{Br}}

\newcommand{\CC}{\mathbb{C}}

\DeclareMathOperator{\Ch}{\mathsf{Ch}}

\DeclareMathOperator{\Cone}{\mathsf{Cone}}

\newcommand{\defeq}{\coloneqq}

\newcommand{\Fqbar}{\overline{\mathbb{F}}_q}

\newcommand{\FT}{\mathsf{FT}}

\DeclareMathOperator{\GL}{GL}

\newcommand{\Gm}{\mathbb{G}_\mathsf{m}}
\newcommand{\gr}{\mathsf{gr}}

\DeclareMathOperator{\Hecke}{\mathsf{H}}

\newcommand{\ho}{\mathop{}\!\mathsf{h}}

\DeclareMathOperator{\Hom}{\mathsf{Hom}}
\newcommand{\homflypt}{\textsc{homfly-pt}}

\newcommand{\hphMod}{\smash{\mhyph\Mod}}

\DeclareMathOperator{\id}{\mathsf{id}}

\DeclareMathOperator{\Ind}{\mathsf{Ind}}

\DeclareMathOperator{\len}{\ell}

\DeclareMathOperator{\Mod}{\mathsf{Mod}}

\DeclareMathOperator{\munit}{\mathsf{1}}

\newcommand{\Negut}{Neguț}

\DeclareMathOperator{\oblv}{\mathsf{oblv}}

\newcommand{\Qlbar}{\QQbar_{\ell}}
\def\QlbarA_#1{\QQbar_{\ell,#1}}
\newcommand{\QQbar}{\overline{\mathbb{Q}}}

\DeclareMathOperator{\rank}{\mathsf{rank}}
\newcommand{\ren}{\mathsf{ren}}

\newcommand{\SBim}{\mathsf{SBim}}

\newcommand{\sfB}{\mathsf{B}}
\newcommand{\sfL}{\mathsf{L}}
\newcommand{\sfR}{\mathsf{R}}
\newcommand{\sfRR}{\mathsf{RR}}

\DeclareMathOperator{\Shv}{\mathsf{Shv}}

\DeclareMathOperator{\Sym}{\mathsf{Sym}}

\DeclareMathOperator{\Vect}{\mathsf{Vect}}

\mathchardef\mhyphensymb="2D
\newcommand{\mhyph}{\mhyphensymb\!}

\newcommand{\arrdisplacementsp}{0.72ex}

\newcommand{\lrangle}[1]{\langle#1\rangle}

\title{Relative Serre duality for Hecke categories}
\author{Quoc P. Ho}
\address{Department of Mathematics, The Hong Kong University of Science and Technology (HKUST), Clear Water Bay, Hong Kong}
\email{phuquocvn@gmail.com}
\author{Penghui Li}
\address{YMSC, Tsinghua University, Beijing, China}
\email{lipenghui@mail.tsinghua.edu.cn}
\date{\today}


\begin{document}
\begin{abstract}
  We prove a conjecture of Gorsky, Hogancamp, Mellit, and Nakagane in the Weyl group case. Namely, we show that the left and right adjoints of the parabolic induction functor between the associated Hecke categories of Soergel bimodules differ by the relative full twist. This exhibits a relative Serre duality pattern for the Hecke categories.
\end{abstract}

\maketitle
\tableofcontents

\section{Introduction}

\subsection{Soergel bimodules}
Let $G$ be a connected reductive group over $\Fqbar$, equipped with a fixed Borel subgroup $B$ and a fixed maximal torus $T \subseteq B$. Let $R \defeq \Sym(X^*(T) \otimes_{\mathbb{Z}} \CC\lrangle{-2})$ be the graded polynomial algebra generated by $\rank T$ elements where the generators live in graded degree $2$. Here, the angular bracket $\lrangle{-}$ denotes a formal grading shift, which is distinct from the square bracket $[-]$ (which will appear later on in the paper) used to denote a cohomological shift.

Consider $\BiMod_R(\Vect^{\gr,\heartsuit})$, the monoidal abelian category of graded bimodules over $R$, where the monoidal product $\otimes_R$ is denoted by $\star$. By construction, $R$ is equipped with an action of the Weyl group $W$ of $G$. The category of Soergel bimodules\footnote{\label{ftn:coxeter_SBim}The notation $\SBim_W$ is slightly abusive as $\SBim_W$ depends on the Coxeter system and not just the Weyl group $W$. Note also that, as the notation might suggest, the category $\SBim_W$ can be more generally defined for any Coxeter system. In this paper, we will only consider the (finite) Weyl group case.} $\SBim_W$ is the full idempotent complete monoidal additive subcategory of $\BiMod_R(\Vect^{\gr,\heartsuit})$ stable under grading shifts and generated by objects of the form $R \otimes_{R^s} R$ where $s \in W$ is a simple reflection. In other words, $\SBim_W$ is generated, under taking finite direct sums, summands, and grading shifts, by $R\otimes_{R^{s_1}} R \otimes_{R^{s_2}} \cdots \otimes_{R^{s_k}} R$ for any sequence of simple reflections $s_i\in W$.

Let $\Ch^b(\SBim_W)$ denote the monoidal $\infty$-category\footnote{See \cref{rmk:infty_vs_triangulated}.} of bounded chain complexes of Soergel bimodules, whose monoidal product is also denoted by $\star$. In~\cite{rouquier_categorification_2006}, Rouquier constructs an object $R_\beta \in \Ch^b(\SBim_W)$, known as the Rouquier complex, associated to each $\beta \in \Br_W$, the corresponding braid group, that is compatible with the braid relations in the sense that we have an equivalence of objects in $\Ch^b(\SBim_W)$
\[
  R_{\beta_1} \star R_{\beta_2} \simeq R_{\beta_1\beta_2}, \qquad\text{for } \beta_1, \beta_2 \in \Br_W. \teq\label{eq:Rouquier_braid_relation}
\]
In particular, we have a complex $\FT_G \in \Ch^b(\SBim_W)$ associated to the full twist braid, i.e., the square of the longest element.

\subsection{Parabolic induction and restriction functors}
Let $P$ be a proper standard parabolic subgroup of $G$ with Levi factor $L$. We have a fully faithful embedding
\[
  \iota: \Ch^b(\SBim_{W_L}) \hookrightarrow \Ch^b(\SBim_W),
\]
induced by the corresponding embedding of the additive categories of Soergel bimodules. One can show that $\iota$ admits both a left and a right adjoint, denoted by $\iota^{\sfL}$ and $\iota^{\sfR}$, respectively.\footnote{See~\cite{gorsky_serre_2019} for more details or \cref{subsect:geometric_parabolic_induction_restriction} below for a geometric perspective.} Since $\iota$ is a monoidal fully faithful embedding, unless confusion is likely to occur, we will identify objects in $\Ch^b(\SBim_{W_L})$ with their images in $\Ch^b(\SBim_W)$ via $\iota$ without explicitly invoking $\iota$.

\subsection{Serre duality for Hecke categories}
Despite their representation theoretic origin, the categories of Soergel bimodules $\Ch^b(\SBim_W)$ in type $A$ play an important role in low dimensional topology. For example, they are originally used in \cites{khovanov_triply-graded_2007} to define the \homflypt{} homology of links and have since attracted a lot of attention in the study of link invariants.

Also in type $A$, and $P=B$, motivated by a certain symmetry in the \homflypt{} homology theory of links, Gorsky, Hogancamp, Mellit, and Nakagane showed that $\iota^{\sfL}$ and $\iota^{\sfR}$ are related to each other by the full twist $\FT_G$. More precisely, they proved the following theorem, which refines some results of~\cite{beilinson_tilting_2004,mazorchuk_projective-injective_2008} in these cases.

\begin{thm}[\cite{gorsky_serre_2019}] \label{thm:original_serre_Hecke}
  For $G = \GL_n$, $P$ the parabolic subgroup given by the partition $(r, 1, 1, \dots, 1)$ (and hence, $L \simeq \GL_r \times \Gm^{n-r}$), we have a natural equivalence of functors $\iota^{\sfR} \simeq \iota^{\sfL}(\FT_{G,L} \star -)$, where $\FT_{G,L} \defeq \FT_L^{-1}\star \FT_G$.
\end{thm}

This result is quite similar to the classical Verdier/Serre duality in algebraic geometry where the two types of pullbacks differ by the dualizing sheaf for a smooth morphism. The authors of~\cite{gorsky_serre_2019} thus refer to this result as a Serre duality for Hecke categories, where $\FT_{G,L}$ plays the role of the dualizing sheaf.

\subsection{Main result}
The main result of this paper is the following theorem, which generalizes \cref{thm:original_serre_Hecke} above to arbitrary connected reductive groups $G$ and arbitrary parabolic subgroups. This was given as \cite[Conjecture 1.8]{gorsky_serre_2019} in their original paper.

\begin{thm} \label{thm:main}
  We have an equivalence of functors $\iota^{\sfR} \simeq \iota^{\sfL}(\FT_{G,L} \star -)$.
\end{thm}

We will prove \cref{thm:main} by geometric means using a geometric avatar of $\Ch^b(\SBim_W)$. We expect that the general strategy of the proof can be adapted directly to the more combinatorial setup of Soergel bimodules. However, the geometric setup allows for a very efficient and transparent proof.

\begin{rmk} \label{rmk:infty_vs_triangulated}
  While the original conjecture in~\cite{gorsky_serre_2019} is formulated in terms of (functors between) triangulated categories, we work exclusively in the $\infty$-categorical setup in this paper. In particular, by default, all categories that appear in this paper are $\infty$-categories.
  
  To recover the triangulated version from the $\infty$-categorical version, one simply passes to the homotopy categories. The fact that adjoint functors between $\infty$-categories induce adjoint functors between their homotopy categories follows from the definition of adjunctions (via $\Hom$-sets/spaces) and the fact that, by definition, for any $\infty$-category $\mathcal{C}$, $\Hom_{\ho\mathcal{C}}(c,d) \simeq \pi_0 \Hom_{\mathcal{C}}(c,d)$, for all $c,d\in \mathcal{C}$, where $\ho\mathcal{C}$ is the homotopy category of $\mathcal{C}$.
\end{rmk}

\section{Geometric Hecke categories}
In this section, we will describe the geometric setup and explain how it is related to the algebraic setup involving Soergel bimodules described in the introduction.

\subsection{The geometric setup}
We define the finite Hecke category associated to $G$ to be
\[
  \Hecke_G \defeq \Shv_{\gr,c}(B\backslash G/B),
\]
where $\Shv_{\gr,c}(B\backslash G/B)$ is the category of graded sheaves on $B\backslash G/B$ developed in~\cite{ho_revisiting_2025}. We recall that the theory of graded sheaves is defined for any Artin stack of finite type and affine stabilizers and is equipped with a six-functor formalism that formally behaves like the classical six-functor formalism for constructible $\ell$-adic sheaves. In fact, for any such stack $\mathcal{Y}$, we have a functor of forgetting the grading
\[
  \oblv_{\gr}: \Shv_{\gr, c}(\mathcal{Y}) \to \Shv_c(\mathcal{Y})
\]
that realizes $\Shv_{\gr, c}(\mathcal{Y})$ as a graded lift of $\Shv_c(\mathcal{Y})$. Moreover, $\oblv_{\gr}$ is compatible with the six-functor formalism on both sides.

In this paper, we only use the formal aspects of the theory of graded sheaves.

\subsubsection{}
$\Hecke_G$ is a monoidal category with respect to the convolution product $\star$. More precisely, if we let $\Vect^{\gr}$ be the symmetric monoidal stable $\infty$-category of graded chain complexes of $\Qlbar$-vector spaces and $\Vect^{\gr,c}$ the full symmetric monoidal subcategory spanned by compact objects, i.e., those with finite-dimensional cohomology, supported in finitely many graded and cohomological degrees, then, $\Hecke_G$ is an algebra object in $\Vect^{\gr,c}\hphMod$, the symmetric monoidal category of $\Vect^{\gr,c}$-module categories. In particular, this means that the convolution product $\star$ is compatible with cohomological shifts $[-]$ and grading shifts $\lrangle{-}$.

When confusion is unlikely to arise, we will drop the $\star$, and write, for example $KL$ instead of $K \star L$ for $K,L\in \Hecke_G$. We have the following result, which allows us to work purely geometrically when studying $\Ch^b(\SBim_W)$.

\begin{thm}[{\cite[\crefnolink{mg:thm:geometric_incarnation_Hecke}]{ho_revisiting_2025}}] \label{thm:geometric_incarnation_Hecke}
  We have an equivalence of monoidal categories\footnote{Since the graded sheaf theory $\Shv_{\gr,c}(-)$ developed in \cite{ho_revisiting_2025} is based on the theory of $\ell$-adic sheaves, it has coefficients in $\Qlbar$. Since $\Qlbar \simeq \CC$ as fields, there is no difference between working over $\mathbb{C}$ and $\Qlbar$.}
  \[
    \Shv_{\gr,c}(B\backslash G/B) \simeq \Ch^b(\SBim_W).
  \]
\end{thm}

In this paper, we will work exclusively in the geometric setting, viewing all objects in $\Ch^b(\SBim_W)$ as objects in $\Hecke_G$.

\begin{rmk}
  The geometric version $\Hecke_G$ plays an important role in low dimensional topology. For example, they are originally used in \cites{webster_geometric_2017} to geometrically define the \homflypt{} homology of links and used in~\cite{ho_graded_2023} to establish a relation between the \homflypt{} link homology and Hilbert schemes of points on $\mathbb{C}^2$, as conjectured by Gorsky, \Negut, and Rasmussen in~\cite{gorsky_flag_2021}. Note also that $\Hecke_G$ is denoted as $\Hecke_G^{\gr}$ in~\cite{ho_graded_2023}.
\end{rmk}

\subsubsection{Standard and co-standard objects}
By the Bruhat decomposition, we have a stratification of $B\backslash G/B$ by $B\backslash BwB/B$ for $w\in W$. Let $\jmath_w: B\backslash BwB/B \to B\backslash G/B$ denote the embedding. The standard (resp. co-standard) object $\Delta_w$ (resp. $\nabla_w$) is defined to be $\jmath_{w!} \Qlbar[\len(w)]\lrangle{\len(w)}$ (resp. $\jmath_{w*} \Qlbar[\len(w)]\lrangle{\len(w)}$).\footnote{There is no half Tate twist involved because, by construction, $\lrangle{-}$ corresponds to half Tate twist.} By construction, we always have a map
\[
  \Delta_w \simeq \jmath_{w!}\Qlbar[\len(w)]\lrangle{\len(w)} \to \jmath_{w*} \Qlbar[\len(w)]\lrangle{\len(w)} \simeq \nabla_w. \teq\label{eq:from_std_to_costd}
\]

Under the equivalence stated in \cref{thm:geometric_incarnation_Hecke}, $\Delta_w$ (resp. $\nabla_w$) corresponds to the Rouquier complex $R_\beta$ associated to the positive (resp. negative) braid $\beta$ associated to $w$. In particular, the full twist element $\FT_G$ corresponds to $\Delta_{w_0}^2$, where $w_0$ is the longest element of $W$.

\begin{rmk} \label{rmk:homotopically_coherent_Rouquier}
  In the geometric and $\infty$-categorical set up, the association $\beta \mapsto R_\beta$ assembles into a monoidal functor $\Br_W \to \Hecke_G$, by, for example,~\cite[Corollary 3.4.2]{tao_homotopical_2021}. This captures all the braid relations homotopically coherently. However, we do not need this stronger statement in this paper, as we require only the existence of the equivalences as written in \cref{eq:Rouquier_braid_relation}.
\end{rmk}

\subsection{Geometric parabolic induction and restriction functors} \label{subsect:geometric_parabolic_induction_restriction}
We will now describe the geometric version of the parabolic induction and restriction functors. Let $P$ be a proper standard parabolic subgroup of $G$ with Levi factor $L$. Let $B_L$ be the Borel subgroup of $L$ defined as the image of $B$ in $L$. Consider the following correspondence
\[
  \begin{tikzcd}
    & B\backslash P/B \ar{dl}[swap]{p} \ar{dr}{q} \\
    B_L\backslash L/B_L && B\backslash G/B.
  \end{tikzcd}
\]

\subsubsection{Parabolic induction functor}
Let $\iota \defeq q_! p^*: \Hecke_L \to \Hecke_G$ denote the functor of parabolic induction. It is easy to see that this is a monoidal functor.

Note that $p^*$ and $p_*$ are inverses of each other
\[
  \begin{tikzcd}
    \Shv_{\gr,c}(B_L\backslash L/B_L)  \ar[description,phantom]{r}{{\scriptstyle\simeq}} \ar[shift left=\arrdisplacementsp]{r}{p^*} & \ar[shift left=\arrdisplacementsp]{l}{p_*} \Shv_{\gr,c}(B\backslash P/B)
  \end{tikzcd}
\]
since $p$ is a bundle with fiber $\sfB U_P$, the classifying space of $U_P$, which is the unipotent radical of $P$. Since $q$ is a closed embedding, $q_! \simeq q_*$ is fully faithful. Thus, $\iota$ also fully faithful. We will therefore frequently view objects of $\Hecke_L$ as objects of $\Hecke_G$ without explicitly invoking the functor $\iota$.

\subsubsection{Parabolic restriction functors}
The functor $\iota$ admits a right adjoint, given by $\iota^{\sfR} \defeq p_* q^!$. Moreover, since $q$ is proper and $p_*$ and $p^*$ are mutually inverses, $\iota$ also admits a left adjoint, given by $\iota^{\sfL} \defeq p_* q^*$. Note that by fully faithfulness, $\iota^{\sfR} \iota \simeq \iota^{\sfL} \iota \simeq \id_{\Hecke_L}$.

\section{The proof}
We are now ready to prove \cref{thm:main}. We will prove the case $L=T$ in \cref{subsec:L=T_case} first, and then deduce the general case from it in \cref{subsec:reduction_to_L=T}. It is interesting to note that the general case follows from the special case $L=T$ via a purely formal argument.

\subsection{The $L=T$ case} \label{subsec:L=T_case}
When $L=T$, \cref{thm:main} is a consequence of \cref{prop:FT_to_FT_L} and the case where $L=T$ of \cref{prop:FT_G_L_induces_equiv_of_cats}.

\begin{prop} \label{prop:FT_G_L_induces_equiv_of_cats}
  The functor $\FT_{G,L} \star -$ induces an equivalence of categories $\ker \iota^{\sfR} \xrightarrow{\simeq} \ker \iota^{\sfL}$.
\end{prop}

\begin{prop} \label{prop:FT_to_FT_L}
  There exists a morphism $\alpha: \FT_G \to \munit$ such that $\iota^{\sfL}(\alpha): \iota^{\sfL}(\FT_G) \to \iota^{\sfL}(\munit) \simeq \munit$ is an equivalence.
\end{prop}

We will now prove \cref{thm:main} for the case where $L=T$, assuming \cref{prop:FT_G_L_induces_equiv_of_cats,prop:FT_to_FT_L}.

\begin{proof}[Proof of \cref{thm:main} for $L=T$]
  % \footnote{\label{ftn:FT_G_to_FT_L_morphism_general}Suppose we can construct a morphism $\alpha: \FT_G \to \FT_L$ such that the induced morphism $\iota^{\sfL}(\alpha): \iota^{\sfL}(\FT_G) \to \iota^{\sfL}(\FT_L)$ is an equivalence, generalizing \cref{prop:FT_to_FT_L}, the proof below would work verbatim for a general $L$. While such an $\alpha$ exists as a consequence of the results proved in this paper, we do not know how to construct it directly as in the case where $L=T$. We thank the anonymous referee for pointing out an error in our original argument, which tried to construct such an $\alpha$ directly.}

  Let $j: B\backslash (G - B) /B \to B\backslash G/B$ denote the complement of the closed immersion $q$. For any $K\in \Hecke_G$, we have the following (co)fiber sequence
  \[
    q_! q^! K \to K \to j_* j^* K
  \]
  which is equivalent to
  \[
    q_! p^* p_* q^! K \to K \to j_* j^* K
  \]
  and hence, to
  \[
    \iota \iota^{\sfR} K \to K \to j_* j^* K.
  \]

  Applying $\FT_{G} \star -$ to the above triangle, we obtain
  \[
    \FT_{G} \iota \iota^{\sfR} K \to \FT_{G} K \to \FT_{G} j_* j^* K.
  \]
  Since $j_* j^* K \in \ker \iota^{\sfR}$, we get $\FT_{G} j_* j^* K \in \ker \iota^{\sfL}$, by \cref{prop:FT_G_L_induces_equiv_of_cats}. Thus, applying $\iota^{\sfL}$ to the above triangle, we obtain
  \[
    \iota^{\sfL}(\FT_{G} K) \simeq \iota^{\sfL}(\FT_{G} \iota \iota^{\sfR} K). \teq\label{eq:equiv_from_exact_triangle}
  \]

  But now,
  \begin{align*}
    \iota^{\sfL}(\FT_{G} \iota\iota^{\sfR} K)
    \simeq \iota^{\sfL}(\FT_G) \iota^{\sfR}(K)
    \xrightarrow[\simeq]{\iota^{\sfL}(\alpha)\iota^{\sfR}(K)} \iota^{\sfR}(K). \teq\label{eq:contracting_i_s}
  \end{align*}
  Here, the first equivalence follows from the fact that $\Hecke_{T}$ is rigid (see~\cite[\crefnolink{gnr:prop:Hecke_rigid}]{ho_graded_2023}), which, by \cite[\crefnolink{mg:cor:lax_implies_strict_rigid}]{ho_revisiting_2025}, implies that $\iota^{\sfL}$ is $\Hecke_T$-linear with respect to the actions of $\Hecke_T$ on itself and on $\Hecke_G$ via $\iota$.\footnote{\label{ftn:rigidity_remark}Strictly speaking, the results cited are stated for the ind-completed (i.e., renormalized) Hecke categories. To use those results, we can, for example, ind-complete all the categories involved, and then restrict back to the full subcategory of compact objects. A similar maneuver is done in \cref{subsec:reduction_to_L=T} below.} The second equivalence follows from \cref{prop:FT_to_FT_L}.

  Combining \cref{eq:equiv_from_exact_triangle,eq:contracting_i_s} and observing that all the morphisms involved are natural in $K$, we obtain the desired equivalence.
\end{proof}

We defer the proofs of \cref{prop:FT_G_L_induces_equiv_of_cats,prop:FT_to_FT_L} to \cref{subsec:proof_prop_FT_G_L_induces_equiv_of_cats,subsec:proof_prop_FT_to_FT_L}, respectively.

\subsection{From $L=T$ to general $L$} \label{subsec:reduction_to_L=T}
We will now deduce \cref{thm:main} for a general $L$ from the case where $L=T$ proved in \cref{subsec:L=T_case}. As we work with multiple Levi's in this subsection, we use $\iota_{L, G}$, $\iota_{L,G}^{\sfL}$, and $\iota_{L,G}^{\sfR}$ to denote the functors previously denoted as $\iota$, $\iota^{\sfL}$ and $\iota^{\sfR}$, respectively.

\subsubsection{Renormalized Hecke categories}
For any reductive group $G$, we define $\Hecke_G^{\ren} \defeq \Ind(\Hecke_G)$ to be the ind-completion of $\Hecke_G$, which is a monoidal category containing $\Hecke_G$ as a full subcategory of compact objects. In fact, $\Hecke_G^{\ren}$ is a compactly generated rigid monoidal category in the sense of~\cite[Volume I, Chapter 1, Definition 9.1.2]{gaitsgory_study_2017} (see~\cite[\crefnolink{gnr:prop:Hecke_rigid}]{ho_graded_2023}).

Ind-extending $\iota_{L, G}$, $\iota_{L, G}^{\sfL}$, and $\iota_{L, G}^{\sfR}$, we obtain functors
\[
  \begin{tikzcd}[column sep=large]
    \Hecke_L^{\ren} \ar[hookrightarrow]{r}[description]{\iota_{L,G}^{\ren}} & \ar[bend left=4*\arrdisplacementsp, shift left=\arrdisplacementsp]{l}{\iota_{L,G}^{\ren, \sfR}} \ar[bend right=4*\arrdisplacementsp, shift right=\arrdisplacementsp]{l}[swap]{\iota_{L,G}^{\ren, \sfL}} \Hecke_G^{\ren}.
  \end{tikzcd}
\]
We still have the same adjunctions as before,\footnote{which justifies, for example, the notation $\iota_{L, G}^{\ren, \sfL}$ (as opposed to $\iota_{L,G}^{\sfL,\ren}$)} and moreover, all of these functors are continuous (i.e., they preserve colimits) and compact preserving. Due to the rigidity of $\Hecke_L^{\ren}$, $\iota_{L, G}^{\ren, \sfL}$ and $\iota_{L,G}^{\ren, \sfR}$ are also $\Hecke_L^{\ren}$-linear (with respect to both left and right actions).

\subsubsection{A further right adjoint} Switching to the renormalized version allows us to invoke the adjoint functor theorem. Namely, $\iota_{L, G}^{\ren, \sfR}$ acquires a further right adjoint $\iota_{L, G}^{\ren, \sfRR}$, which is also continuous, since $\iota_{L,G}^{\ren, \sfR}$ preserves compact objects. Moreover, $\iota_{L,G}^{\ren, \sfRR}$ is also $\Hecke_L^{\ren}$-linear due to the rigidity of $\Hecke_L^{\ren}$. Thus, $\iota_{L,G}^{\ren, \sfRR}$ is determined by where it sends the unit object $\munit$. More precisely, we have the following equivalence of functors
\[
  \iota_{L,G}^{\ren, \sfRR} \simeq \iota_{L,G}^{\ren, \sfRR}(\munit) \star \iota_{L, G}^{\ren}(-).
\]

\begin{lem}
  We have an equivalence of functors $\iota_{T,G}^{\ren, \sfRR} \simeq \FT_G^{-1} \star \iota_{T, G}^{\ren}(-)$. Equivalently, $\iota_{T,G}^{\ren, \sfRR}(\munit) \simeq \FT_G^{-1}$.
\end{lem}
\begin{proof}
  By the case where $L=T$ of \cref{thm:main} proved in \cref{subsec:L=T_case}, we have an equivalence of functors
  \[
    \iota_{T,G}^{\ren, \sfR} \simeq \iota_{T,G}^{\ren, \sfL}(\FT_G \star -).
  \]
  Taking right adjoints on both sides, using the fact that the right adjoint of $\iota_{T,G}^{\ren, \sfL}$ is $\iota_{T,G}^{\ren}$, we obtain
  \[
    \iota_{T,G}^{\ren, \sfRR} \simeq \FT_G^{-1} \star \iota_{T,G}^{\ren}(-).
  \]
\end{proof}

\begin{cor} \label{cor:computing_iota_L_G_ren_RR}
  $\iota_{L,G}^{\ren, \sfRR}(\munit) \simeq \FT_{G,L}^{-1}$, and hence, we have an equivalence of functors $\iota_{L,G}^{\ren, \sfRR} \simeq \FT_{G,L}^{-1} \star \iota_{L,G}^{\ren}(-)$.
\end{cor}
\begin{proof}
  The second part follows from the first part as discussed above. For the first part, observe that
  \[
    \iota_{L,G}^{\ren, \sfRR}(\munit) \star \FT_L^{-1}
    \simeq \iota_{L,G}^{\ren, \sfRR}(\FT_L^{-1})
    \simeq \iota_{L,G}^{\ren, \sfRR}(\iota_{T,L}^{\ren, \sfRR}(\munit))
    \simeq \iota_{T,G}^{\ren, \sfRR}(\munit)
    \simeq \FT_G^{-1}.
  \]
  Thus, $\iota_{L,G}^{\ren, \sfRR}(\munit) \simeq \FT_G^{-1} \star \FT_L \simeq \FT_{G,L}^{-1}$.
\end{proof}

\subsubsection{Concluding the proof of \cref{thm:main}} Taking left adjoints on both sides of the equivalence in \cref{cor:computing_iota_L_G_ren_RR}, we obtain an equivalence of functors
\[
  \iota_{L,G}^{\ren, \sfR} \simeq \iota_{L,G}^{\ren, \sfL}(\FT_{G,L} \star -).
\]
Since all functors involved preserve compact objects, restricting to the full subcategories of compact objects, we obtain the desired equivalence of functors $\iota_{L,G}^{\sfR} \simeq \iota_{L,G}^{\sfL}(\FT_{G,L} \star -)$.
\qed

\subsection{Proof of \cref{prop:FT_G_L_induces_equiv_of_cats}} \label{subsec:proof_prop_FT_G_L_induces_equiv_of_cats}
We start with the following lemma.

\begin{lem} \label{lem:FT_G_L_vs_FT_G_equiv_cat}
  The following statements are equivalent:
  \begin{enumerate}
    \item \label{item:FT_G_L_equiv_cat} $\FT_{G,L} \star -$ induces an equivalence of categories $\ker \iota^{\sfR} \xrightarrow{\simeq} \ker \iota^{\sfL}$.
    \item \label{item:FT_G_equiv_cat} $\FT_{G} \star - $ induces an equivalence of categories $\ker \iota^{\sfR} \xrightarrow{\simeq} \ker \iota^{\sfL}$.
  \end{enumerate}
\end{lem}
\begin{proof}
  We will show that \ref{item:FT_G_equiv_cat} implies \ref{item:FT_G_L_equiv_cat}. The converse is similar.

  By rigidity of $\Hecke_L$, which implies that $\iota^{\sfL}$ and $\iota^{\sfR}$ are $\Hecke_L$-linear, we see that $\FT_L \star -$ induces an equivalence of categories between $\ker \iota^{\sfR}$ (resp. $\ker \iota^{\sfL}$) with itself. But now, this observation, combined with the current assumption that $\FT_G \star -$ induces an equivalence between $\ker \iota^{\sfR}$ and $\ker \iota^{\sfL}$, allows us to conclude since
  \[
    \FT_{G,L} \star - \simeq \FT_{L}^{-1} \star \FT_G\star -.
  \]
\end{proof}

It remains to prove \cref{lem:FT_G_L_vs_FT_G_equiv_cat}.\ref{item:FT_G_equiv_cat}.

\begin{lem} \label{lem:FT_G_equiv_cat}
  $\FT_G \star -$ induces an equivalence of categories $\ker \iota^{\sfR} \xrightarrow{\simeq} \ker\iota^{\sfL}$.
\end{lem}
\begin{proof}
  Observe that
  \[
    \ker \iota^{\sfR} = \lrangle{\nabla_u \mid u \notin W_L},
  \]
  where the RHS denotes the smallest $\Vect^{\gr,c}$-linear full stable $\infty$-subcategory  of $\Hecke_G$ containing all objects of the form $\nabla_u$ for $u \notin W_L$. In other words, it is the smallest full subcategory of $\Hecke_G$ containing $\nabla_u$ for $u\notin W_L$ that is closed under finite direct sums, shifts, cones, and grading shifts.

  Since for any $u\in W$,
  \[
    \len(w_0u) + \len(u^{-1}) = \len(w_0) - \len(u) + \len(u^{-1}) = \len(w_0),
  \]
  $\Delta_{w_0} \simeq \Delta_{w_0u} \Delta_{u^{-1}} \simeq \Delta_{w_0u} (\nabla_u)^{-1}$, or equivalently,
  \[
    \Delta_{w_0} \nabla_u \simeq \Delta_{w_0u}. \teq\label{eq:Delta_w_0_nabla_u}
  \]
  Thus,
  \[
    \Delta_{w_0} \ker \iota^{\sfR} = \lrangle{\Delta_{w_0u} \mid u \notin W_L} = \lrangle{\Delta_t \mid t\in \tau}
  \]
  where
  \[
    \tau = \{w_0 u \mid u \notin W_L\} \subset W.
  \]

  Observe that $\tau$ is closed with respect to the Bruhat order on $W$, or equivalently, the union of $B\backslash BvB/B$ for $v\in \tau$ is a closed substack of $B\backslash G/B$. Indeed, since the map $W \to W$ given by $u \mapsto w_0 u$ reverses the Bruhat order (see~\cite[Example 3, p. 119]{humphreys_reflection_1990}), to show that $\tau$ is closed, it suffices to show that $W \setminus W_L$ is open. But this is equivalent to the fact that $W_L$ is closed, which is true.

  Thus,
  \[
    \Delta_{w_0} \ker \iota^{\sfR}
    = \lrangle{\Delta_t \mid t\in \tau}
    = \lrangle{\nabla_t \mid t\in \tau}
    = \lrangle{\nabla_{w_0u} \mid u \notin W_L}.
  \]
  But now, we have
  \begin{align*}
    \Delta_{w_0}^2 \ker \iota^{\sfR}
    &= \lrangle{\Delta_{w_0} \nabla_{w_0u} \mid u\notin W_L} \\
    &= \lrangle{\Delta_{w_0^2 u} \mid u\notin W_L} \\
    &= \lrangle{\Delta_{u} \mid u\notin W_L} \\
    &= \ker \iota^{\sfL},
  \end{align*}
  and we are done.
\end{proof}

The proof of \cref{prop:FT_G_L_induces_equiv_of_cats} is complete.
\qed

\subsection{Proof of \cref{prop:FT_to_FT_L}} \label{subsec:proof_prop_FT_to_FT_L}
\subsubsection{The morphism $\alpha$}
The map $\alpha: \FT_G \to \munit$ is defined to be $\FT_G = \Delta_{w_0} \Delta_{w_0} \to \Delta_{w_0} \nabla_{w_0} \simeq \munit$, where the non-trivial morphism is given by \cref{eq:from_std_to_costd}.

\subsubsection{$\iota^{\sfL}(\alpha)$ is an equivalence}
To show that $\iota^{\sfL}(\alpha)$ is an equivalence, it suffices to show that $\Cone(\alpha) \in \ker \iota^{\sfL}$, which is equivalent to showing that
\[
  \Delta_{w_0} \Cone(\Delta_{w_0} \to \nabla_{w_0}) \in \ker \iota^{\sfL}.
\]

This is a direct consequence of the following two lemmas and the fact that $i^* j_! \simeq 0$ where $i$ and $j$ form a pair of closed embedding and open complement.

\begin{lem} \label{lem:no_intersection_with_W_L}
  $\{w_0 v \mid v \neq w_0\} \cap \{1\} = \emptyset$.
\end{lem}
\begin{proof}
  This is clear since $w_0$ is its own inverse.
\end{proof}

\begin{lem} \label{lem:shifted_cone_computation}
  $\Delta_{w_0} \Cone(\Delta_{w_0} \to \nabla_{w_0}) \in \lrangle{\Delta_{w_0 v} \mid w_0 \neq v}$.
\end{lem}
\begin{proof}
  By construction,
  \[
    \Cone(\Delta_{w_0} \to \nabla_{w_0}) \in \lrangle{\Delta_v \mid w_0 > v} = \lrangle{\nabla_v \mid w_0 > v} = \lrangle{\nabla_v \mid w_0 \neq v},
  \]
  where the first equality is due to the fact that $W \setminus \{w_0\}$ is closed with respect to the Bruhat order on $W$. Thus,
  \[
    \Delta_{w_0} \Cone(\Delta_{w_0} \to \nabla_{w_0}) \in \lrangle{\Delta_{w_0} \nabla_v \mid w_0 \neq v} = \lrangle{\Delta_{w_0v} \mid w_0 \neq v},
  \]
  where the last equality is due to \cref{eq:Delta_w_0_nabla_u}.
\end{proof}

% We will prove these two lemmas in the remainder of this subsection. For the proof of \cref{lem:shifted_cone_computation}, we need the following elementary lemma.

% \begin{lem} \label{lem:nabla_times_Delta}
%   For $x, y \in W$ such that $y > x$, we have
%   \[
%     \nabla_{y^{-1}} \Delta_x =\nabla_{y^{-1}x}.
%   \]
% \end{lem}
% \begin{proof}
%   Since $y>x$, we can write $y = xz$ such that $\len(y) = \len(x) + \len(z)$ and hence, $y^{-1} = z^{-1} x^{-1}$, where $\len(y^{-1}) = \len(z^{-1}) + \len(x^{-1})$. Thus,
%   \begin{align*}
%     \nabla_{y^{-1}} \Delta_x
%     \simeq \nabla_{z^{-1}} \nabla_{x^{-1}} \Delta_x
%     \simeq \nabla_{z^{-1}}
%     \simeq \nabla_{y^{-1}x}.
%   \end{align*}
% \end{proof}

The proof of \cref{prop:FT_to_FT_L} is now complete.\qed

\section*{Acknowledgements}
We thank the anonymous referees for their careful reading of the manuscript, providing extensive feedback which greatly helps improve the clarity of the exposition, and catching an important error in an earlier version of the paper. Q. Ho and P. Li are partially supported by the National Key R\&D Program of China (Grant 2024YFA1014700) and the Hong Kong RGC GRF grants 16304923 and 16301324.

\printbibliography
\end{document}